\documentclass{amsart}
\usepackage{amscd}
\usepackage{latexsym,amssymb,palatino,color,graphicx,amsmath,amsthm, hyperref,enumerate}
\usepackage{listings}
\usepackage{xcolor}
\lstdefinestyle{pythonstyle}{
	language=Python,
	backgroundcolor=\color{gray!10},
	basicstyle=\ttfamily\footnotesize,
	keywordstyle=\color{blue},
	commentstyle=\color{green!60!black},
	stringstyle=\color{red!75!green},
	showstringspaces=false,
	breaklines=true,
	captionpos=b
}

\newtheorem{theorem}{Theorem}[section]
\newtheorem{lemma}[theorem]{Lemma}

\newtheorem{proposition}[theorem]{Proposition}

\theoremstyle{definition}

\theoremstyle{remark}

\numberwithin{equation}{section}

\newcommand{\psl}[2]{\operatorname{\rm PSL}_#1(#2)}

\usepackage{longtable}

\begin{document}
	
\title[Generating $\psl{2}{q}$ by elements of prime orders $2$ and $p$]{Generating $\psl{2}{q}$ by elements of prime orders $2$ and $p$}
	
\author[D.~Farenick, R.~Maleki, S.~Medina Varela, S.~Singla]{Douglas Farenick, Roghayeh Maleki, Sofia Medina Varela, and Sushil Singla}
\address{Department of Mathematics and Statistics, University of Regina, Regina, Saskatchewan S4S 0A2, Canada}
	
\email{douglas.farenick@uregina.ca}
\email{roghayeh.maleki@uregina.ca}
\email{smc472@uregina.ca}
\email{sushil@uregina.ca}
	
	
\begin{abstract} For primes $p\geq 5$, we determine those $q$ for which 
the projective special linear group $\psl{2}{q}$ over the finite
field of order $q$ is $(2,p)$-generated---that is, 
there exist two elements of $\psl{2}{q}$ of orders $2$ and $p$, respectively, that generate
$\psl{2}{q}$.
\end{abstract}
	
\maketitle

\section{Introduction}

There is a large body of work concerning the generation of finite groups by two elements, 
including the well-known theorems (i) that every finite non-abelian simple group is $2$-generated
\cite{aschbacher-guralnick1984,miller1901,steinberg1962} and (ii) that
the probability $p(G)$ that two randomly selected elements of a finite group $G$ will generate $G$ satisfies
$p(G)\rightarrow 1$ as $|G|\rightarrow\infty$ \cite{kantor--lubotzky1990,liebeck--shalev1995}. 
It is often important to know some special types of generating pairs,
such as requiring generators to be of some prescribed order, or that generators be reflective of some geometric property. It is the former requirement
that concerns us in this paper.

A group $G$ is said to be $(a,b)$-generated if there are two elements in $G$ 
of orders $a$ and $b$, respectively, that generate $G$.
The only finite groups generated by two involutions ({\it i.e.}, elements of order $2$) are the dihedral groups, 
while every non-abelian finite simple group is generated by an involution and an element of 
prime order \cite{king2017}. In such cases, the smallest pair $(2,p)$, 
with $p$ prime, of interest is the case of $p=3$, for which it is 
known that there are $2^{\aleph_0}$ isomorphism classes of simple $(2,3)$-generated groups \cite{mason--pride1990}.
Among the known $(2,3)$-generated groups are: the alternating groups $A_n$, for $n\geq 9$ \cite{miller1901};
the sporadic simple groups, with the exception of four of them \cite{wolder1989}; and the projective special  linear groups $\psl{2}{q}$,
for $q\neq 9$, and $\psl{3}{q}$, for $q\neq 4$.

With primes $p\geq 5$, the literature concerning $(2,p)$-generated groups is not as rich as that for 
$(2,3)$-generated groups. Therefore, the purpose of this paper is to augment our knowledge by giving a complete characterisation of 
the $(2,p)$-generated projective special linear groups $\psl{2}{q}$, for primes $p\geq 5$. 
There are no assumed relations between $q$, which is necessarily a power of a prime, and $p$.

Our main result is:

\begin{theorem}\label{main-theorem-1}
Let $p\geq 5$ be a fixed prime number, and suppose 
$q\geq 4$ is a power of a prime number. The group
$\mbox{\rm PSL}_2(q)$ is $(2,p)$-generated if and only if one of the following conditions is satisfied:
	\begin{enumerate}
		\item[(i)] $q\equiv0\pmod p$;
		\item[(ii)] $q\equiv1\pmod p$; or
		\item[(iii)] $q\equiv-1\pmod p$.
	\end{enumerate}
\end{theorem}


\section{Proof of Theorem \ref{main-theorem-1}}

We first settle on some notation. Throughout, $\mathbb F_q$ denotes a finite field with $q$ elements; thus, $q$ is necessarily a power of some prime number.
Therefore, we also assume throughout that $q$ is a power of a prime.
The group of invertible $2\times2$ matrices with entries in $\mathbb{F}_q$ is
 $\operatorname{GL}_2(q)$, and $\operatorname{SL}_2(q)$ denotes 
 the subgroup of $\operatorname{GL}_2(q)$ consisting of matrices with determinant equal to $1$. 
 We shall use the notation $\psl{2}{q}$ for the quotient group of $\operatorname{SL}_2(q)$ by its centre, which is $\{\pm I\}$. Note if $q$ is even, then $\psl{2}{q}$  agrees with the group $\operatorname{SL}_2(q)$ itself. The order of $\psl{2}{q}$ is given by
\begin{align*}
	\frac{q(q-1)(q+1)}{\gcd(2, q-1)}.
\end{align*}
Note that, the elements of $\psl{2}{q}$ are cosets of the center of $\operatorname{SL}_2(q)$. However, throughout the article by abuse of notation, we will view them as matrices, while keeping in mind that the matrices $A$ and $-A$ are equal in $\psl{2}{q}$. 
Two matrices $A, B\in \psl{2}{q}$ are said to be conjugate if there exists an invertible matrix $P\in \psl{2}{q}$ such that $PAP^{-1}=B$. We mention a known fact that for any prime power $q$, any element in $\psl{2}{q}$ is conjugate to one of the following four elements: \begin{equation}\label{conjugate_q_odd} \begin{pmatrix}
	r& 0\\
	0 &r^{-1}
\end{pmatrix}, \begin{pmatrix}
	s & 1\\
	0 &s
\end{pmatrix},
\begin{pmatrix}
	s & \Delta\\
	0 &s
\end{pmatrix}, \begin{pmatrix}
0 & -1\\
1 &\theta+\theta^q
\end{pmatrix}
,\end{equation} where $r,s\in\mathbb F_q^*$ and $s=\pm 1$,  $\Delta$ is a non-square element of $\mathbb F_q$ (i.e. $\Delta\in\mathbb F_q\setminus\{x^2 : x\in\mathbb F_q\}$), and $\theta\in \mathbb{F}_{q^2}\setminus\mathbb F_q$ with the condition that $\theta^q=\theta^{-1}$. This has been proved for $\operatorname{SL}_2(q)$ in \cite{james--liebeck2001} (also see \cite[Remark 2.1]{adan-bante--harris2012}). But,
if $A$ and $B$ are conjugate in $\operatorname{SL}_2(q)$, then their images are also conjugate in $\psl{2}{q}$, consequently the conjugacy classes of elements is $\operatorname{SL}_2(q)$ and $\psl{2}{q}$ are essentially the same. 

We remark that if $q$ is a power of $2$, then $\mathbb F_q$ does not contain any non-square element, 
and so, in this case, any element in $\psl{2}{q}$ is conjugate to only three types of matrices as mentioned in \eqref{conjugate_q_odd},
with the exception of $\begin{pmatrix}
	s & \Delta\\
	0 & s
\end{pmatrix}$, where $s=\pm 1$. But, for the uniformity of the arguments, we will consider all four cases and deal with even and odd powers
together, with the knowledge that $\mathbb F_q$ does not contain any non-square element if $q$ is a power of $2$.

Recall the only groups which are $(2,2)$-generated are dihedral groups, and 
$\mbox{\rm PSL}_2(q)$ is $(2,3)$-generated if and only if $q\neq 9$. 
Therefore, we will work with a fixed prime $p\geq 5$ in the rest of this section. There is no assumed relationship between $q$ and $p$, 
and $q$ may be the power of a prime other than $p$.

If $\psl{2}{q}$ contains at least one element of order $p$, then, by Cauchy's theorem, $p$ must divide $\frac{q(q-1)(q+1)}{\gcd(2, q-1)}$. 
Consequently, it follows that $p \mid q$, $p \mid (q-1)$, or $p \mid (q+1)$. These cases correspond exactly to 
$q \equiv 0 \pmod{p}$, $q \equiv 1 \pmod{p}$, or $q \equiv -1 \pmod{p}$, respectively. Our main result, Theorem~\ref{main-theorem-1}, is that these 
necessary conditions on $q$ 
are also sufficient for the $(2,p)$-generation of $\psl{2}{q}$, leading to a complete determination of
the $(2,p)$-generated groups, for primes $p\geq 5$, among the projective special linear groups $\mbox{\rm PSL}_2(q)$.

The proof of Theorem \ref{main-theorem-1} is in two parts: the first part for certain low
values of $q$, and the second part for all other values of $q$ (namely $q=8$ and $q\geq 16$).
We treat the pertinent lower values of $q$ first.

\begin{proposition}\label{prop1}
	Let $q \in \{3,4,5,7,9,11,13\}$. The group $\psl{2}{q}$ is $(2,p)$-generated, for a prime $p\geq 5$, if and only if  
	\begin{enumerate}
		\item[(i)] $q\equiv0\pmod p$;
		\item[(ii)] $q\equiv1\pmod p$; or
		\item[(iii)] $q\equiv-1\pmod p$.
	\end{enumerate}
\end{proposition}

\begin{proof}
	We already observed above that if $\psl{2}{q}$ is  $(2,p)$-generated, then $p$ must be an odd prime divisor of the order of $\psl{2}{q}$, so $q\equiv0\pmod p$, $q\equiv1\pmod p$ or $q\equiv-1\pmod p$. Conversely, we show that for every $q \in \{3,4,5,7,9,11,13\}$ and for each odd prime divisor $p\ge 5$ satisfying (i)-(iii), $\psl{2}{q}$ is $(2,p)$-generated by consideration of the following cases.
		
		Note that the order of $\psl{2}{3}$ is $12$, so the only odd prime divisor is $3$, and it cannot be $(2,p)$-generated for any prime $p\geq 5$. 
		
		The groups $\psl{2}{4} \cong A_5$ has 
		a presentation $\langle a,b| a^2 = b^5 = (ba)^3 = 1 \rangle$, so it is $(2,5)$-generated. Since the largest order of element in $A_5$ is $5$, so $\psl{2}{4}$ is not $(2,p)$-generated for $p>5$.
		Similarly, $\psl{2}{5}\cong A_5$ is $(2,5)$-generated, and is not $(2,p)$-generated for $p>5$.
		
		The group $\psl{2}{7}$ is $(2,7)$-generated \cite{behr--mennicke1968}. Moreover, we note that the order of $\psl{2}{7}$ is $168$, 
		which has only $3$ and $7$ as the odd prime divisors; hence, for $p\geq 5$, the group $\psl{2}{7}$ is $(2,p)$-generated only for $p=7$.
		
		The group $\psl{2}{9}\cong A_6$ and  has a  presentation
		\begin{align*}
			\left\langle a,b\ | a^2 = b^4 = (ab)^5 = (ab^2)^5 = 1 \right\rangle,
		\end{align*}
		so it is clearly $(2,5)$-generated. Furthermore, $\psl{2}{9}$  has order $360$, which
		has only $3$ and $5$ as the the odd prime divisors, and so $\psl{2}{9}$  is not $(2,p)$-generated for any prime $p>5$.
		
		The prime divisors of the order of $\psl{2}{11}$ are $3$, $5$, and $11$; therefore, for $p\geq 5$, $\psl{2}{11}$ can be $(2,p)$-generated only if $p=5$ or $p=11$. This proves the necessity of the conditions. By \cite{behr--mennicke1968}, the group $\psl{2}{11}$ has presentation
		\begin{align*}
			\left\langle a,b\ | a^2 = b^{11} = (ba)^3 = (b^2ab^6a)^3 = 1 \right\rangle,
		\end{align*}
		where we can choose $a=\begin{pmatrix}
			0&-1\\1&0
		\end{pmatrix}$ and $b=\begin{pmatrix}
		1&1\\0&1
	\end{pmatrix}$. So, $\psl{2}{11}$ is $(2,11)$-generated. To prove the sufficiency of conditions for $\psl{2}{11}$, it remains to prove that it is $(2,5)$-generated. To do this, we let 
	\[
	G = \psl{2}{11}=\left\langle a,b\ | a^2 = b^{11} = (ba)^3 = (b^2ab^6a)^3 = 1 \right\rangle
	\]
	 and $H = \langle a,b^3a \rangle$ with $a$ and $b$ as defined above. We note that $(b^3a)a = b^3$, and since $b$ has order $11$, we deduce that $\langle b \rangle = \langle b^3\rangle \leq \langle a,b^3a \rangle = H$. Therefore, $H = G$.  Moreover, the element $b^3a=\begin{pmatrix}
	3&-1\\1&0
\end{pmatrix}$ has order $5$, so $G$ is $(2,5)$-generated. 
		
		The prime divisors of the order of $\psl{2}{13}$ are $3$, $7$, and $13$, and
		so for $p\geq 5$, $\psl{2}{13}$ can be $(2,p)$-generated only for $p=7$ and $p=13$. 
		By the presentation of $\psl{2}{13}$ given in \cite{behr--mennicke1968}, $\psl{2}{13}$ is clearly $(2,13)$-generated. 
		The $(2,7)$-generation of $\psl{2}{13}$ follows from \cite[Theorem 8]{macbeath1969}.

These cases above establish the claim.
\end{proof}

To prove our main result for all remaining values of $q$, we begin with the following lemmas.

\begin{lemma}\label{prop} Let $p\geq 5$ be a fixed prime and suppose 
$q \equiv 0 \pmod{p}$, $q \equiv 1 \pmod{p}$, or $q \equiv -1 \pmod{p}$. If $A\in\psl{2}{q}$ is of order $p$, then, up to conjugation, $A$ has one of the
following forms:
	\begin{enumerate}
		
			\item[(a)] if $q\equiv 0 \pmod{p}$, then  \begin{align*}
				A\in
				\left\{\begin{pmatrix}
					1 & 1\\
					0 &1
				\end{pmatrix},
				\begin{pmatrix}
					1 & \Delta\\
					0 &1
				\end{pmatrix}\right\},
			\end{align*}
			where $\Delta$ is a non-square in $\mathbb{F}_{q}$;
			
		\item[(b)]  if $q\equiv 1 \pmod{p}$, then 
		\begin{align*}
			A=
			\begin{pmatrix}
				r& 0\\
				0 &r^{-1}
			\end{pmatrix},
		\end{align*}
		where $r\in \mathbb{F}_q$ such that $r^p = 1$;
		
		\item[(c)] if $q\equiv -1 \pmod{p}$, then 
		\begin{align*}
		A=
		\begin{pmatrix}
			0 & -1\\
			1 &\theta+\theta^{-1}
		\end{pmatrix}
		,
	\end{align*}
		where $\theta\in \mathbb{F}_{q^2}\setminus\mathbb F_q$, $\theta+\theta^{-1}\in\mathbb F_q$, and $\theta^{q}=\theta^{-1}$. 
		
	\end{enumerate}
\end{lemma}
\begin{proof} 
	With reference to \eqref{conjugate_q_odd} and because the order of $A\in\psl{2}{q}$ is an odd prime $p$, $A$ is conjugate to one of the following matrices:
	
	{\centering
		\begin{equation}\label{eq3}
			\left\{\begin{pmatrix}
				1 & 1\\
				0 &1
			\end{pmatrix},
			\begin{pmatrix}
				1 & \Delta\\
				0 &1
			\end{pmatrix}\right\}, \begin{pmatrix}
				r& 0\\
				0 &r^{-1}
			\end{pmatrix},  \, \mbox{or}\,\begin{pmatrix}
				0 & -1\\
				1 &\theta+\theta^q
			\end{pmatrix},
	\end{equation}}
where $\Delta$ is a non-square in $\mathbb{F}_{q}$, $r\in \mathbb{F}_q$ such that $r^p = 1$, and  $\theta\in \mathbb{F}_{q^2}\setminus\mathbb F_q$
satisfies $\theta+\theta^{-1}\in\mathbb F_q$ and $\theta^{q}=\theta^{-1}$.  
	
	We now compute the order of each type of matrices in \eqref{eq3}.
	
	\textbf{Case (i)}. Assume $A\in
	\left\{\begin{pmatrix}
		1 & 1\\
		0 &1
	\end{pmatrix},
	\begin{pmatrix}
		1 & \Delta\\
		0 &1
	\end{pmatrix}\right\}$. We note that $$\begin{pmatrix}
		1 & 1\\
		0 &1
	\end{pmatrix}^p = \begin{pmatrix}
	1 & p\\
	0 &1
	\end{pmatrix}\quad\text{ and }\quad\begin{pmatrix}
	1 & \Delta\\
	0 &1
	\end{pmatrix}^p=\begin{pmatrix}
	1 & p\Delta\\
	0 &1
	\end{pmatrix}.$$  Thus, if $A$  is of order $p$, then in the first case $p=0\in\mathbb F_q$, and in the second case $p\Delta=0\in\mathbb F_q$. 
	But, as $\Delta$ is non-square, we have $p=0$ in both cases, which is only possible in the case $q\equiv 0 \pmod{p}$. 
	
	\textbf{Case (ii)}. If $A=\begin{pmatrix}
		r & 0\\
		0 & r^{-1}
	\end{pmatrix}$, then the order of $A$ is $p$ if and only if the order of $r$ is $p$. Since $r\in\mathbb F_q^*$, we also know that the order of $r$ must divide $q-1$; thus, $p$ must divide $q-1$, which is only possible in the case $q\equiv 1 \pmod{p}$. 
	
	\textbf{Case (iii)}. Suppose 
			\[
			A=\begin{pmatrix}
				0 & -1\\
				1 & \theta+\theta^{q}
			\end{pmatrix},
			\]
			where $\theta \in \mathbb{F}_{q^2}\setminus\mathbb F_q$ satisfies $\theta^{q+1}=1$. We note that the minimal polynomial of $A$ is
			\[
			m_A(t)=t^2-(\theta+\theta^{q})t+1.
			\]
			This also shows that the eigenvalues of $A$ are $\theta$ and $\theta^q$. Furthermore, since $\theta$ and $\theta^{q}$ are roots of $t^{q+1}-1$, it follows that
			\[
			m_A(t)\mid (t^{q+1}-1).
			\]
			So, there exists a polynomial $g(t)\in \mathbb{F}_q[t]$ such that $A^{q+1}-I=g(A)(A^2-(\theta+\theta^{q})A+I)$. However, $(A^2-(\theta+\theta^{q})A+I)=0$ as it is the minimal polynomial of $A$, and so $A^{q+1}=I$. Hence, if the matrix $A$ has order $p$, then we must have that $p | (q+1)$ which is only possible in the case when $q\equiv -1 \pmod{p}$.

	Finally, we note that the Cases (i)-(iii) gives the required conclusions in (a)-(c).  This completes the proof.
\end{proof}

For any prime power $q$, we recall that the non-zero elements of $\mathbb{F}_q$ form a multiplicative group, which is cyclic of order $q-1$. A non-zero element of $\mathbb{F}_q$ is primitive if it is a generator of this multiplicative group. 
The following lemma will also turn out to be useful for our purposes.

\begin{lemma}\label{number_generators}
	Except $q=2,3,4,5,7,9,11,13$, the number of primitive elements in $\mathbb F_q$ is at least $5$.
\end{lemma}
\begin{proof}
	The number of primitive elements in  $\mathbb{F}_q$ is equal to the number of generators of the group of unity of the field, which is cyclic of order $q-1$. Therefore, the number of primitive elements of $\mathbb{F}_q$ is  $\phi(q-1)$, where $\phi$ denotes the Euler Totient function.	We have $\phi(8-1) = 6$.
	Next, we show that $\phi(q-1)\geq5$, for $q\geq 16$ (equivalently $q>13$) through the following cases and we will be done.
\begin{enumerate}[(1)]
	\item Let $q=2^n$, where $n\geq4$. Then, $q-1=2^n-1$.
	\begin{enumerate}[(a)]
		\item If $2^n-1$ is a prime number, then
		\begin{equation*}
			\phi(q-1)=\phi(2^n-1)=(2^n-1)-1=2^n-2\geq5,
		\end{equation*}
		since $n\geq 4$.
		\item If $2^n-1$ is a composite number, then there exist odd prime numbers $3\leq p_1<p_2<\ldots< p_k$ for some $k\geq 1$ and positive integers $n_1,n_2,\ldots,n_k$ such that
		\begin{equation*}
			q-1=2^n-1=p_1^{n_1}p_2^{n_2}\ldots p_k^{n_k}.
		\end{equation*}
		Assume that $k=1$, i.e., $q-1=2^n-1=p_1^{n_1}$. We note that $n_1\geq 2$ since $q-1$ is composite. Since $n>1$ and $n_1>1$, the equation 
			\begin{align*}
				2^n-p_1^{n_1}=1,
			\end{align*} 
			does not admit any solution in $p_1$ due to the famous  Catalan's Conjecture which is now known as Mihăilescu's theorem. In order for $2^n-p_1^{n_1}=1$ to have a solution in $p_1$, one of $n$ and $n_1$ must be equal to $1$, which is impossible in this case since $n\geq4$ and $n_1\geq 2$. 
			
			Thus, $k\geq2$, i.e.,  $q-1=2^n-1=p_1^{n_1}p_2^{n_2}\ldots p_k^{n_k}$, where $3\leq p_1<p_2<\ldots <p_k$. Recall that the Euler totient $\phi$ is multiplicative for coprime numbers, that is, if $u$ and $v$ are coprime, then $\phi(uv) = \phi(u)\phi(v)$. Consequently, we have
			
			$$\phi(q-1) = \phi(p_1^{n_1})\phi(p_2^{n_2}) \ldots \phi(p_k^{n_k})\geq \phi(p_1^{n_1})\phi(p_2^{n_2}).$$
			
			Furthermore, since $\phi(p^n)\geq p^{n-1}(p-1)$ for all primes $p$ and $n\geq 1$, we get $$\phi(q-1) \geq (p_1-1)(p_2-1).$$
			Since $p_1\geq 3$ and $p_2\geq 5$, we have $$\phi(q-1) \geq 2\times 4\geq 8.$$
			
			Therefore, if $k\geq2$, then we have $\phi(q-1)\geq5$.
	\end{enumerate}
	Hence, when $q\geq 16$ is even, we have that $\phi(q-1)\geq5$.
	\item Assume that $q = s^l \geq 16$ for some  odd prime $s$ and an integer $l\geq 1$. Then, $q-1$ is even
	\begin{align*}
		q-1=2^{n_0}p_1^{n_1}p_2^{n_2}\ldots p_k^{n_k}, 
	\end{align*}
	for some odd primes $3\leq p_1<p_2<\ldots < p_k$, an integer $n_0\geq 1$, and non-negative integers $n_1,\ldots,n_k$. Now, we consider the following cases.
	\begin{enumerate}
		\item Assume first that $n_1 = n_2 = \ldots = n_k = 0$. Then, $q-1 = s^l-1 =2^{n_0}$, which implies that 
		 \begin{align*}
			s^l-2^{n_0}=1.
		\end{align*} 
		
		 If $l,n_0>1$, then by Mihăilescu's theorem, the only solution to $s^l-2^{n_0}=1$ is achieved when $n_0=3$ and $s=3$, and $l=2$. In other words, $q-1=2^3 = 8$, and thus $q=9$, which is not considered here as $q\geq16$.
		
		 Therefore, one of $n_0$ or $\ell$ is equal to $1$. If $n_0=1$, then $q-1=2^{n_0}=2$  meaning that $q=3$, which is excluded by assumption. If $l=1$, then $q-1=s-1=2^{n_0}$ which implies that 
		 \begin{align*}
		 	s=2^{n_0}+1,
		 \end{align*} 
		is a Fermat prime. For the Fermat number  $2^{n_0}+1$ to be a prime, the integer $n_0$ must be a power of $2$, so there exists an integer $m\geq 0$ such that $s = 2^{2^m}+1$. Since $q=s=2^{2^m}+1 \geq 16$, we deduce that $m\geq 2$. Hence, we have
		 \begin{align*}
		 	\phi(q-1) = \phi(s-1)=2^{2^m-1}(2-1)=2^{2^m-1}\geq 8.
		 \end{align*}
		 
		
	\item Assume that one of $n_1,n_2,\ldots,n_k$ is non-zero. Assume without loss of generality that $n_1\geq 1$. Then, we have
	\begin{align*}
		\phi(q-1)&=\phi(2^{n_0})\phi(p_1^{n_1})\phi(p_2^{n_2})\ldots \phi(p_k^{n_k}) \\
		&\geq \phi(2^{n_0})\phi(p_1^{n_1})= 2^{n_0-1}p_1^{n_1-1}(p_1-1).
	\end{align*}
	If $n_0,n_1>1$, then $\phi(q-1)\geq2\times3\times2=12$. 
	
	If $n_0=n_1=1$, then  we distinguish the cases whether $k\geq2$ or $k=1$.
	If $k\geq2$, then $\phi(q-1)\geq1\times1\times(p_1-1)p_2^{n_2-1}(p_2-1)\geq8$, as $p_2\geq5$. Now, if $k=1$, then we have that $q-1=2p_1$, and so $q=2p_1+1$. Therefore,  for $p_1=3$ or $p_1=5$ we have $q=7$ or $q=11$,  respectively, which are both excluded by assumption. Moreover, for $p_1\geq7$ we always have that $\phi(q-1)\geq 2^{n_0-1}p_1^{n_1-1}(p_1-1)\geq1\times1\times6\geq6 $. 
	
	If $n_0\geq3$ and $n_1=1$, then $\phi(q-1)\geq4\times2=8$.
	If $n_0 = 2$ and $n_1=1$, then  we distinguish the cases whether $p_1=3$ or $p_1\geq 5$.  If $p_1\geq5$, then $\phi(q-1)\geq2\times4=8$. If $p_1=3$, then $q-1=2^2\times 3=12$ and so $q=13$, which is excluded by assumption. 
	If $n_0=1$ and $n_1\geq2$, then $$\phi(q-1)\geq p_1^{n_1-1}(p_1-1) \geq  p_1(p_1-1)\geq 6.$$
	Therefore, again we have that $\phi(q-1)\geq5$.
		\end{enumerate}
\end{enumerate} 
Hence, we have that $\phi(q-1)\geq5$ in all cases. This completes the proof.
\end{proof}

The next proposition completes the proof of Theorem \ref{main-theorem-1}.

\begin{proposition}\label{main-theorem-part2}
Let $p\geq 5$ be a fixed prime number, and suppose $q$ is a prime power such that $q=8$ or $q\ge16$.
The group
$\mbox{\rm PSL}_2(q)$ is $(2,p)$-generated if and only if one of the following conditions is satisfied:
	\begin{enumerate}
		\item[(i)] $q\equiv0\pmod p$;
		\item[(ii)] $q\equiv1\pmod p$; or
		\item[(iii)] $q\equiv-1\pmod p$.
	\end{enumerate}
\end{proposition}
\begin{proof} As we noted earlier, we need only prove the sufficiency of the conditions. Recall that if $q=8$ or
$q\geq 16$, then Lemma \ref{number_generators} asserts the number of primitive elements in $\mathbb F_q$ is at least $5$. 
	
	We proceed by constructing an explicit pair of generators to show that $\psl{2}{q}$ is $(2,p)$-generated. 
	Note that it is sufficient to show that there exist elements $g, h \in \psl{2}{q}$ with $|g|=2$ and $|h|=p$ such that the subgroup generated by $g$ and $h$, $H = \langle g, h \rangle$, is not contained in any maximal subgroup of $\psl{2}{q}$. 
	
	For cases (i)-(iii), since we also have $q\neq 2,3$,  using the classification theorem \cite[Corollary~2.2]{king2005} (also see \cite{dickson-book}, \cite[\S~5.2]{fulton-harris-book} for details),  the maximal subgroups of $\psl{2}{q}$ are among one of the following groups:
	\begin{enumerate}
		\item Borel subgroups---that is, any subgroup of order $\frac{q(q-1)}{\gcd(2,q-1)}$ that is conjugate to the subgroup of upper triangular matrices in $\psl{2}{q}$;
		\item dihedral groups $D_{\frac{2(q \pm 1)}{\gcd(2,q-1)}}$; \label{dihedral}
		\item the exceptional groups $A_4, S_4$, or $A_5$;
		\item subgroups $\mbox{\rm PSL}_2(\mathbb {F}_{q'})$, where $q = (q')^t$ for a prime $t$;
		\item subgroups $\mbox{\rm PGL}_2(\mathbb {F}_{q'})$, where $q = (q')^2$.
	\end{enumerate}
	
\noindent\textbf{Case (i):} The assumption is that $q \equiv 0\pmod{p} $. 

By Lemma~\ref{prop} (a), an element of order $p$ in $\psl{2}{q}$ is conjugate to $
		\begin{pmatrix}
			1 & 1\\
			0 &1
		\end{pmatrix}$ or
		$\begin{pmatrix}
			1 & \Delta\\
			0 &1
		\end{pmatrix},$ where $\Delta$ is a non-square in $\mathbb F_q$. Define 
\[
g=\begin{pmatrix}
	0 & 1\\
	-1 & 0
\end{pmatrix} \quad
\mbox{and} \quad
h=\begin{pmatrix}
	1& b\\
	0 & 1
\end{pmatrix},
\]
where $b$ is a primitive element of $\mathbb{F}_q$ such that $b^4-3b^2+1\neq 0$.
The existence of such $b$ is proved as follows. Since the number of primitive elements in $\mathbb F_q$ in our cases is at least $5$, so we can  choose $b$ in $\mathbb{F}_q$ to be
an appropriate generator that does not satisfy the fourth degree equation $b^4-3b^2+1= 0$. 


We now show $H=\langle g,h\rangle$
is not contained in any of the five maximal subgroups (1)-(5) listed above.


\begin{enumerate}
	\item  We claim that $H=\langle g,h\rangle$ is not contained in a Borel subgroup of $\mathrm{PSL}_2(q)$. We prove this claim by contradiction. Assume that $H$ is contained in a Borel subgroup. Recall that a subgroup of $\mathrm{PSL}_2(q)$ is contained in a Borel subgroup if and only if is conjugate to a subgroup of the upper triangular matrices in $\psl{2}{q}$.
	For $H$, this means that it leaves a one-dimensional subspace of $\mathbb F_q^2$ invariant. 
	
	The matrix $h$ is an upper triangular matrix, and therefore preserves the one-dimensional subspace $\operatorname{Span}\{e_1\}$, and this is the only subspace that $h$ leaves invariant.
	Consequently, $H=\langle g,h\rangle$ must leave $\operatorname{Span}\{e_1\}$ 
	 invariant.		
	However, $g$
	does not preserve this subspace. So, $\langle g,h\rangle$ does not preserve any one-dimensional subspace, which contradicts the fact that it is a subgroup of a Borel subgroup. We deduce by contradiction that $H=\langle g,h\rangle$ is not a subgroup of any Borel subgroup.
	
	\item In the dihedral group $D_{2n}$, every element is either an element of the cyclic subgroup of order $n$, or an involution outside of this cyclic subgroup. Additionally, if $n$ is even, then there is a unique involution in the cyclic subgroup of order $n$, which is central in $D_{2n}$. 
	
	If $\langle g,h\rangle$ is a subgroup of a dihedral group of $\psl{2}{q}$, then $h$ lies in the cyclic part, and we either have $ghg^{-1} = h^{-1}$ (when $g$ is outside of the cyclic group) or $ghg^{-1} = h$ (when $g$ is the central involution of the cyclic group). However, since $b\neq 0$, we have
	\begin{align*}
		ghg^{-1} = 
	\begin{pmatrix}
		1 & 0 \\
		-b & 1
		\end{pmatrix}
\neq\begin{pmatrix}
	1& b\\
	0 & 1
	\end{pmatrix} =h	\end{align*}
and
	\begin{align*}
	ghg^{-1} = 
	\begin{pmatrix}
		1 & 0 \\
		-b & 1
	\end{pmatrix}
	\neq\begin{pmatrix}
		1& -b\\
		0 & 1
	\end{pmatrix} =h^{-1}.
	\end{align*}
Therefore, $\langle g,h\rangle$ is not contained in a dihedral subgroup of $\psl{2}{q}$.
	
	\item An element of the group $S_4$ has cycle type $(1,1,1,1), (2,1,1),(2,2),(3,1)$, or $(4)$. Therefore, the group $S_4$, and by extension its
	subgroup $A_4$, cannot have elements of order $p\geq 5$. 
	Similarly, an element of the group $A_5$ has cycle type $(1,1,1,1,1)$,$(2,2,1)$, $(3,1,1),$ and $(5)$. Hence, $A_5$ has an element of prime order $p\geq 5$ only when $p = 5$. The group $A_5$ is $(2,5)$-generated, and  
	if $x$ and $y$ are generators of $A_5$ such that $|x| = 2$ and $|y| = 5$, then $|xy| = |yx| \in \{3,5\}$.
	To see this, note $|xy|=|yx|$ because $yx=x^2yx=x(xy)x^{-1}$. Because  
	the elements of $S_5$ of order $4$ are $4$-cycles,
	$A_5$ has no elements of order $4$, which implies $|xy|\not=4$.   
	If it were true that $|xy|=2$, then we would have $xy=y^{-1}x$, and so the subgroup $\langle x, y\rangle \cong D_{10}$, the dihredral
	group of order $10$. Therefore, as obviously $|xy|\not=1$,
	the only possible orders of $xy$ and $yx$ are $3$ or $5$. 
	
	By the previous paragraph, $H$ cannot be contained in any subgroup of $\psl{2}{q}$ isomorphic to $A_4$ or $S_4$.
	Now, if we assume by contradiction that $H$ is contained in a subgroup of $\psl{2}{q}$ that is isomorphic to $A_5$, then $hg$ has order $3$ or $5$. We have 
	\begin{align*}
		&(hg)^3
		=
		\begin{pmatrix}
			-b^3+2b & b^2-1\\
			-b^2+1&b
		\end{pmatrix},\\
		&(hg)^5
		=
		\begin{pmatrix}
			-b^5+4b^3-3b & b^4-3b^2+1\\
			-b^4+3b^2-1&b^3-2b
		\end{pmatrix}.
	\end{align*}
	If $(hg)^3$ is equal to the identity element, then $b \in\{1,-1\}$ which is impossible by choice of $b$. If $(hg)^5$ is the identity, then $b^4-3b^2+1 = 0$ which is again impossible by choice of $b$. Consequently, $H$ is not contained in any subgroup of $\psl{2}{q}$ isomorphic to $A_5$.

	\item Since $b$ is a primitive element of $\mathbb F_q$, $b$ does not belong to any subfield of $\mathbb{F}_q$.  
	Therefore, we have that $h\notin \mbox{\rm PSL}_2(\mathbb {F}_{q'})$, where $q = (q')^t$ for a prime $t$,
	and consequently $H$ is not contained in the subgroup 
	$\mbox{\rm PSL}_2(\mathbb {F}_{q'})$.
	
	\item The proof is identical to case (4) above.
	
\end{enumerate}

\textbf{Case (ii):} The assumption for this case is $q \equiv 1\pmod{p} $. By Lemma~\ref{prop} (b), an element of order $p$ is conjugate to $
\begin{pmatrix}
	r& 0\\
	0 &r^{-1}
\end{pmatrix},$	where $r\in \mathbb{F}_q$ such that $r^p = \pm 1$. Fix a primitive element $\omega$ of $\mathbb{F}_q$. Define 
\[
g=\begin{pmatrix}
	0 & 1\\
	-1 & 0
\end{pmatrix} \quad
\mbox{and} \quad
h=\begin{pmatrix}
	r& b\\
	0 & r^{-1}
\end{pmatrix},
\]
where $r=\omega^{\frac{q-1}{p}}$ and $b\in \mathbb{F}_q$ is a primitive element such that $b^4-3b^2+1 \neq 0$. 

Note that the element $h$ has order $p$, regardless of the choice of $b$. The analysis that $H=\langle g,h\rangle$ is not contained in any maximal subgroup in the list (1)-(5) is identical to the 
analysis in Case (i), with the exception of the group $A_5$. Assume that $H$ is contained in a subgroup of $\psl{2}{q}$ which is isomorphic to $A_5$. Then, as we saw in Case (i), the element $hg$ must have order $3$ or $5$. Using the expressions
\begin{align*}
	&(hg)^3
	=
\begin{pmatrix}
	-b^3+2b & r(b^2-1)\\
	-(b^2-1)r^{-1}&b
\end{pmatrix}, \mbox{ and }\\
	&(hg)^5
	=
	\begin{pmatrix}
		-b^5+4b^3-3b & r(b^4-3b^2+1)\\
		r^{-1}(-b^4+3b^2-1)&b^3-2b
	\end{pmatrix},
\end{align*}
we deduce that $b\in \{1,-1\}$ or $b^4-3b^2+1 = 0$. By choice of $b$, neither of these identities hold. Consequently, $H$ is not contained in any subgroup of $\psl{2}{q}$ isomorphic to $A_5$.

\textbf{Case (iii):} The assumption in this case is $q \equiv -1\pmod{p} $.  By Lemma~\ref{prop}(c), an element of order $p$ is conjugate to $
\begin{pmatrix}
	0 & -1\\
	1 &\theta+\theta^{-1}
\end{pmatrix}$,
where $\theta\in \mathbb{F}_{q^2}\setminus\mathbb F_q$ and $\theta^q=\theta^{-1}$. 
Define 
\[
g=\begin{pmatrix}
	0 & 1\\
	-1 & 0
\end{pmatrix} \quad
\mbox{and} \quad
h=\begin{pmatrix}
	0& -b^{-1}\\
	b & \theta+\theta^{-1}
\end{pmatrix},
\]
where $b\in \mathbb{F}_q^*$ is a primitive element such that $b^4 \neq 1$ and $b^6 \neq 1$. Let us first show that such an element $b$ always exist. If $q \not \in\{5,7\}$, then for any primitive element $b\in \mathbb{F}_q$, we have $b^{4}\neq 1$ and $b^6 \neq 1$. Since we are dealing with cases $q \not \in\{5,7\}$, so we can always assume the existence of a primitive element such that $b^{4}\neq 1$ and $b^6 \neq 1$.


The required calculations differ somewhat from the previous two cases, and so, for completeness, we provide full details of the arguments. As before,
our aim is to show $H=\langle g,h\rangle$ is not contained in any of the maximal subgroups listed in (1)-(5).

\begin{enumerate}
	\item  Contrary to what we aim to prove, assume  
	$H$ is contained in a Borel subgroup. Because
	a subgroup of $\mathrm{PSL}_2(\mathbb F_q)$ is contained in a Borel subgroup if and only if it is reducible, 
	the group $H$ leaves a one-dimensional subspace of $\mathbb F_q^2$ invariant. However, 
	$h$ does not admit eigenvalues in $\mathbb{F}_q$, and so it cannot preserve any one-dimensional subspace.
	Thus, it is not possible for $H$ to be a subgroup of a Borel subgroup. 
		
	\item
		If $\langle g,h\rangle$ is a subgroup of a dihedral group of $\psl{2}{q}$, then $h$ lies in the cyclic part, and we either have $ghg^{-1} = h^{-1}$ (when $g$ is outside of the cyclic group) or $ghg^{-1} = h$ (when $g$ is the central involution of the cyclic group). However,
	\begin{align*}
		ghg^{-1} &= 
	\begin{pmatrix}
			\theta + \theta^{q} & -b \\
			{b}^{-1} & 0
		\end{pmatrix}
		\neq\begin{pmatrix}
			0& -b^{-1}\\
			b & \theta+\theta^{-1}
		\end{pmatrix} =h,
	\end{align*}
	and
	\begin{align*}
		ghg^{-1} = 
		\begin{pmatrix}
			\theta + \theta^{q} & -b \\
			{b}^{-1} & 0
		\end{pmatrix}
		\neq\begin{pmatrix}
			\theta+\theta^{-1}& b^{-1}\\
			-b & 0
		\end{pmatrix} =h^{-1},
	\end{align*}
	 since  $b^4 \neq 1.$
Therefore, $\langle g,h\rangle$ is not contained in a dihedral subgroup of $\psl{2}{q}$.
	
	\item  As we saw before, since $H$ admits an element of order $p\geq 5$, the subgroups $A_4$ and $S_4$ cannot contain $H$. If $H$ is contained in a subgroup of $\psl{2}{q}$ isomorphic to $A_5$, then $p = 5$ and $hg$ has order $3$ or $5$. The element $hg\in H$ is
	\begin{align*}
		hg =
		\begin{pmatrix}
			b^{-1} & 0\\
			-(\theta+\theta^q)&b
		\end{pmatrix}.
	\end{align*}
	If $(hg)^k$ is equal to the identity of $\psl{2}{q}$ for $k\in \{3,5\}$, then 	$b^{2k} = 1$. By choice of $b$, we  know that $b^6 \neq 1$. Moreover, $b^{10}\neq 1$ since $q\equiv -1 \pmod 5$. Consequently, $hg$ cannot have order equal to $3$ or $5$, thus contradicting the fact that $H$ is a subgroup of $A_5$.
	
	\item Note that $g\in \mbox{\rm PSL}_2(\mathbb {F}_{q'})$, where $q = (q')^t$ for a prime $t$, but, by assumption, 
	$b $ does not belong to any subfield of $\mathbb{F}_q$. Therefore, 
	we have that $h\notin \mbox{\rm PSL}_2(\mathbb {F}_{q'})$, 
	and consequently $H$ is not contained in $\mbox{\rm PSL}_2(\mathbb {F}_{q'})$.
	
	\item The proof is identical to case (4) above.
\end{enumerate}

This concludes the proof that $H = \langle g, h \rangle$ is not contained in any maximal subgroup of $\psl{2}{q}$, and so 
it must be that $H=\psl{2}{q}$. Hence, $\psl{2}{q}$ is $(2,p)$-generated.
\end{proof}


\section{Discussion}

The problem of determining generators $x$ and $y$ of finite groups $G$ having prescribed orders has a long history. Dating back to 1893, Hurwitz showed that 
the conformal automorphism group of an algebraic curve of genus $g \geq  2$ is at most $84(g - 1)$, and that the upper bound is achieved if and only if the automorphism group
is generated by elements $x$ and $y$ for which $x$, $y$, and $xy$ are of order $2$, $3$, and $7$, respectively. Such groups are known as Hurwitz groups, and 
Macbeath determined that $\psl{2}{q}$ is a Hurwitz group if and only if 
$q=7$ or $q\equiv \pm 1\pmod  7$ \cite{macbeath1969}. 

In contrast to these classical results,
the main result of this paper determines precisely those prime powers $q$ for which $\psl{2}{q}$ is $(2,p)$-generated, for a fixed prime number $p\geq 5$ (unrelated to $q$),
but without placing any further requirements on the order of the product $xy$ of generators $x$ and $y$ (of orders $2$ and $p$). Moreover, the proof of Theorem \ref{main-theorem-1} has constructive aspects whereby 
generators of prescribed order are explicitly identified. Potential applications of Theorem \ref{main-theorem-1} to problems in combinatorics or geometry remain to be considered.

\section*{Acknowledgements}
This work was supported, in part, by the Discovery Grant program of
the Natural Sciences and Engineering Research Council of Canada and 
the Pacific Institute for the Mathematical Sciences Postdoctoral Fellowship program.


\end{document}